\documentclass[11pt]{amsart}
\usepackage[latin1]{inputenc}
\usepackage{amsmath, amsfonts, amssymb, amsthm, mathtools, mathrsfs,xypic,bm}
\usepackage{enumerate}
\usepackage[numbers]{natbib}
\usepackage{setspace}
\usepackage[usenames,dvipsnames]{color}

\usepackage[pagebackref,colorlinks,linkcolor=red,citecolor=blue,urlcolor=blue,hypertexnames=true]{hyperref}
\usepackage[pagebackref,hypertexnames=true]{hyperref}

\newtheorem{theorem}{Theorem}[section]
\newtheorem{lemma}[theorem]{Lemma}
\newtheorem{proposition}[theorem]{Proposition}
\newtheorem{corollary}[theorem]{Corollary}

\theoremstyle{definition}

\newtheorem{remark}[theorem]{Remark}

\newcommand{\G}{\mathcal{G}}

\newcommand{\T}{\mathbb{T}}
\newcommand{\Z}{\mathbb{Z}}

\newcommand{\R}{\mathbb{R}}
\newcommand{\C}{\mathbb{C}}

\newcommand{\K}{\mathcal{K}}	

\newcommand{\rank}{\mathrm{rank}}
\newcommand{\Aut}{\mathrm{Aut}}

\newcommand{\Inn}{\mathrm{Inn}}

\begin{document}
\title { Connective Bieberbach groups}
\author{Marius Dadarlat}\address{MD: Department of Mathematics, Purdue University, West Lafayette, IN 47907, USA}\email{mdd@purdue.edu}
\author{Ellen Weld}\address{EW:
   Department of Mathematics, Purdue University, West Lafayette, IN 47907, USA}\email{weld@purdue.edu}
  \thanks{M.D. was partially supported by an NSF grant \#DMS--1700086 } 	

\begin{abstract}
We prove that
a  Bieberbach group  with trivial center  is not connective and use this property to show that
a Bieberbach group is connective if and only if it is poly-$\Z$.
\end{abstract}
\maketitle


\section{Introduction}
 Connectivity \cite{Dad-Pennig-connective} is a homotopy invariant property of a separable $C^*$-algebra $A$
 that  has three interesting consequences: absence of nonzero projections, quasidiagonality, and realization of the Kasparov groups as homotopy classes of asymptotic morphisms from $A$ to $B\otimes \K$  without  suspensions, that is $KK(A,B)\cong [[A,B\otimes \K]]$, if $A$ is nuclear.
 A separable amenable (nuclear)  C*-algebra is connective if it embeds in the C*-algebra $\prod_n B_n /\bigoplus _n B_n$ where $B_n=C_0((0,1], M_n(\C))$ is the C*-algebra of continuous functions from $[0,1]$ to $n\times n$ complex matrices vanishing at $0$.
 A countable discrete group $G$ is called connective if  the kernel $I(G)$ of the trivial representation $\iota:C^*(G)\to \C$ is a connective $C^*$-algebra.
 If $G$ is connective and amenable, then $K^0(C^*(G))=\Z [\iota] \oplus [[I(G),\K]]$. This implies that the nontrivial part of the $K$-homology of $C^*(G)$ can be realized as homotopy classes of asymptotic representations $\{\pi_t:G \to U(\infty)\}_{t\in [1,\infty)}$ with  $t\mapsto \pi_t(g)$ continuous, $g \in G$, and
 $\lim_{t\to \infty}\|\pi_t(g_1g_2)-\pi_t(g_1)\pi_t(g_2)\|=0$, $g_1,g_2\in G$.
Large classes of amenable connective  groups   were exhibited in \cite{Dad-Pennig-connective}, \cite{Dad-Pennig-homotopy-symm}, \cite{Dad-Pennig-Schneider}.

 Connectivity of a separable nuclear $C^*$-algebra $A$ can be characterized solely in terms of its primitive spectrum $\mathrm{Prim}(A)$.
 It was shown in \cite{Dad-Pennig-connective} that the existence of a
nonempty compact open subset of $\mathrm{Prim}(A)$ is an obstruction to connectivity.
By a remarkable result of J. Gabe \cite{Gabe:connectivity} this is the only obstruction.

A crystallographic group of dimension $k\geq 1$ is a discrete co-compact subgroup of the isometry group
$\mathrm{Iso}(\R^k)=\R^k\rtimes O(k)$ of the Euclidean space $\R^k$. In his renowned work on Hilbert's 18th Problem, Bieberbach proved that any crystallographic group $G$ of dimension $k$ fits into an exact sequence
\begin{equation}\label{eq:bieber}
\xymatrix{
	1 \ar[r] & N \ar[r] & G \ar[r] & D \ar[r] & 1
}
\end{equation}
where $N\cong\Z^k$ is a maximally abelian  subgroup of $G$, called the lattice of $G$, and $D$ is a finite group called the holonomy group.
Moreover, for each fixed $k$, there are only finitely many isomorphism classes of crystallographic
groups of dimension $k$ and two crystallographic groups are isomorphic if and only
if they are conjugate in the group $\R^k\rtimes GL_k(\R)$,  \cite{Charlap}.
A torsion free crystallographic group
is called a Bieberbach group. The orbit space $\R^k/G$ of a Bieberbach group is a $k$-dimensional closed flat Riemannian manifold $M$ with holonomy group isomorphic to $D$.

Which Bieberbach groups are connective? It was shown in \cite{Dad-Pennig-connective} that while any Bieberbach group with cyclic holonomy is connective \cite[Thm. 3.8]{Dad-Pennig-connective},
the celebrated Hantzsche-Wendt group \cite{Hantzsche-Wendt} (denoted here by $\Gamma$) is not connective \cite[Cor. 3.2]{Dad-Pennig-connective}.
$\Gamma$ is generated by two elements $x$ and $y$ subject to two relations:
\(x^2yx^2=y,\quad y^2xy^2=x,\)  and fits into an exact sequence
\[
\xymatrix{
	1 \ar[r] & \Z^3 \ar[r] & \Gamma \ar[r] & \Z/2\oplus \Z/2 \ar[r] & 1.
}
\]
The non-connectivity of $\Gamma$ was proved by showing that $\widehat{\Gamma}\setminus \{\iota\}$ is a compact-open subset of the unitary spectrum $\widehat{\Gamma}$.
There are exactly 10 non-isomorphic 3-dimensional Bieberbach groups. The Hantzsche-Wendt group $\Gamma$ is singled out among these groups by the property that  it has finite  first homology group, in fact, isomorphic to  $\Z/4\oplus \Z/4$ \cite{Wolf-spaces}.
The first homology group $H_1(G,\Z)$ of a group $G$ can be computed as $H_1(G,\Z)=G/[G,G]$ where $[G,G]$ is the commutator subgroup of $G$.
In view of the lack of  connectivity of the Hantzsche-Wendt group, it is natural to ask to what extent connectivity
of a Bieberbach group relates to its homology. Independently of us, Szczepa\'{n}ski asks the same question at the end of his paper \cite{Szczepanski}. As a key step in our study of connectivity of Bieberbach groups we prove the following:
\begin{theorem}\label{thm:a}
  A Bieberbach group with finite first homology group is not connective.
\end{theorem}

By work of Calabi \cite{Calabi} (see the discussion in \cite{Szczepanski-book}, \cite{Wolf-spaces}) any Bieberbach group $G$ with infinite first homology group can be written as
  an iterated semidirect product
 \begin{equation}\label{eqn:convention} G\cong ((H \rtimes \Z)\rtimes \cdots )\rtimes \Z\end{equation}
 where either $H$ is a Bieberbach group such that $H_1(H,\Z)$ is finite, or $H=\{1\}$, in which case $G$ is a poly-$\Z$ group.
We use the Calabi decomposition of $G$  in conjunction with Theorem~\ref{thm:a} to prove  the following:
 \begin{theorem}\label{thm:char}
   Let $G$ be a Bieberbach group. The following assertions are equivalent.
   \begin{itemize}
   \item[(i) ] $G$ is connective.
\item[(ii) ]  Every  nontrivial subgroup of $G$ has a nontrivial center.
\item[(iii) ] $G$ is a poly-$\Z$ group.
\item[(iv) ] $\widehat{G}\setminus \{\iota\}$ has no nonempty compact open subsets.
\end{itemize}
\end{theorem}
Auslander and Kuranishi \cite{Auslander-Kuranishi} showed that any finite group is the holonomy group of some torsion free Bieberbach group.
In general one cannot determine if a Bieberbach group is connective by just looking at its holonomy group.
Nevertheless, by using results of \cite{Hiller} and \cite{Diffuse} one can derive the following:
\begin{corollary}\label{cor} Let $D$ be a finite group.
\begin{itemize}
   \item[(a) ] If $D$ is not solvable, then any Bieberbach group with holonomy $D$ is not connective.
\item[(b) ]  If $D$ is solvable with all Sylow subgroups cyclic (solvability is automatic in this case),  then any Bieberbach group with holonomy $D$ is connective.
\item[(c) ]   If $D$ is solvable and has a non-cyclic Sylow subgroup, then there are Bieberbach groups $G_1$ connective and $G_2$ not connective both with holonomy group $D$.
\end{itemize}
\end{corollary}
A discrete group $G$ is called locally indicable  if  every finitely generated non-trivial subgroup
 $L$ of $G$ has a quotient isomorphic to $\Z$ or equivalently $H_1(L,\Z)$  is infinite.
 The group $G$ is called diffuse if every non-empty finite subset $A$ of $G$ has an element $a\in  A$  such that for any
$g\in G$, either $ga$ or $g^{-1}a$ is not in $A$, \cite{Bowditch}.  Linnell and Witte Morris \cite{Linnell-Witte} proved that an amenable group is diffuse if and only if is locally indicable. It follows then by
Theorem~\ref{thm:char} that   a Bieberbach group $G$ is connective $\Leftrightarrow$ $G$ is locally indicable $\Leftrightarrow$ $G$ is diffuse.
\section{Bieberbach groups with trivial center}\label{sec:2}
The purpose of this section is to prove Theorem~\ref{thm:a}.

Hiller and Sah \cite{Hiller} showed that a finite group $D$ is primitive, meaning that it can be realized as the holonomy group of a Bieberbach group  with finite first homology group if and only if no cyclic Sylow subgroup of $D$  admits a normal complement.
For instance cyclic groups (including the trivial group) are not primitive but $(\Z/p)^{m}$ is primitive if $m\geq 2$ and  $p$ is prime.
Recall that as a consequence of Burnside's normal $p$-complement theorem,  if $p$ is the smallest prime that divides the order of a finite group,  then any  cyclic $p$-Sylow subgroup admits a normal complement \cite[p.138, 6.2.12]{Scott}. Therefore, if all Sylow subgroups of a finite group $D$ are cyclic, then $D$
 is not primitive.

For a Bieberbach group $G$ the following conditions are equivalent, see
\cite[Prop. 1.4]{Hiller}:
\begin{itemize}
\item[(i) ] $H_1(G,\Z)$ is a finite group.
\item[(ii) ] $G$ has trivial center, $Z(G)=\{0\}$.
\item[(iii) ] The action of $G$ by conjugation on its lattice subgroup  $N$  has exactly one fixed point, $N^G=\{0\}$.
\end{itemize}
For the sake of completeness we revisit the proof of these equivalences [cf. \cite{Hiller}].
The rank of a discrete abelian group $L$,  denoted  $\rank(L)$, is the dimension of the $\R$-vector space $L\otimes_{\Z} \R$.
\begin{proposition}\label{pro:revisit} If $G$ is a Bieberbach group, then
$\rank(H_1(G,\Z))=\rank(Z(G)) = \rank (N^G).$
\end{proposition}
\begin{proof}
 One observes that $Z(G)=N^G$ since $N$ is maximal abelian and $G$ acts on $N$ by conjugation.
By\cite [Cor.  6.4]{Brown:book-cohomology} the exact sequence \eqref{eq:bieber} induces an exact sequence
\[
\xymatrix{
	H_2(D,\Z) \ar[r] & H_1(N,\Z)^{D} \ar[r] & H_1(G,\Z) \ar[r] & H_1(D,\Z) \ar[r] & 0.
}
\]
The action of $D$ on $H_1(N,\Z)=N/[N,N]=N$ is induced by the conjugation action of $G$ on $N$ so that
$H_1(N,\Z)^{D}\cong N^G$. Since $D$ is a finite group, so are the groups $H_1(D,\Z)$ and $H_2(D,\Z)$.
It follows that $\rank(H_1(G,\Z))= \rank (N^G)$.
\end{proof}
$G$ acts on $N$ by conjugation: $h \mapsto ghg^{-1}$, $h\in N$. Since $N$ is abelian, this action
descends to an action of $D$ on $N=\Z^k$ by automorphisms. We denote the corresponding representation by $\theta:D \to GL_k(\Z)=\mathrm{Aut}(\Z^k).$ The map $\theta$ is injective since $N$ is maximal abelian.

By duality $G$ acts on $\widehat{N}\cong \mathbb{T}^k$ by automorphisms: $g\cdot \chi = \chi(g^{-1}\cdot g)$.
This action descends to an action of $D\cong G/N$ which is the dual of the action $\theta$ discussed above and is denoted by $\theta^*:D \to \mathrm{Aut}(\T^k).$
Let  us consider the fixed points of these actions and observe that $(\Z^k)^D$  is a subgroup of $\Z^k$
and that $(\T^k)^D$ is a closed subgroup of $\T^k$ and  hence it must a be Lie subgroup by Cartan's theorem.
If  $\Gamma$ is the Hantzsche-Wendt group,
the corresponding representation $\theta: D=\Z/2\oplus \Z/2 \to GL_3(\Z)$
has the property that $(\Z^3)^D=\{\mathbf{0}\}$ while $(\T^3)^D$ consists of eight points $(\pm 1, \pm 1, \pm 1)$, see for example \cite{Dad-Pennig-connective}. The finiteness of  $(\T^3)^D$ is not coincidental.
In fact we are going to see as a consequence of Proposition \ref{prop}  that the conditions $(i),(ii), (iii)$ from above are equivalent to the following condition which  plays a crucial role in the proof of our main result.
\begin{itemize}
\item[(iv) ] The dual action of $G$ on $\widehat{N}$ has finitely many fixed points i.e.  $\widehat{N}^G$ is a finite group.
\end{itemize}


 The rank of a Lie group $K$, denoted $\rank(K),$  is the dimension of any one of its Cartan subgroups. If $K$ is  abelian, then $\rank(K)$ coincides with the dimension of $K$.

\begin{proposition}\label{prop} If $D\subset GL_k(\Z)$ is a finite group, then
 $\rank(\Z^k)^D=\rank(\T^k)^D$. Therefore, if $G$ is a Bieberbach group, then $\rank(H_1(G,\Z))=\rank(Z(G)) = \rank (N^G)=\rank(\widehat{N}^G).$\end{proposition}

\begin{proof} Write $\theta:D \to GL_k(\Z)$ for the representation defined by the inclusion map from the statement.
For $s\in D$, $\theta(s)$ is given by an $k\times k$ matrix $A(s)$ with integer coefficients. Thus $D$ acts on $\Z^k$ by $\mathbf{v} \mapsto A(s) \mathbf{v}$. The dual action maps a character $\chi:\Z^k \to \T$ to $\chi(A(s^{-1})\,\cdot\,)$.

By definition
$\mathbf{H}:=(\Z^k)^D=\{\mathbf{v}\in \Z^k: A(s)\mathbf{v}=\mathbf{v}, s\in D\}.$ Since $\mathbf{H}$ is a subgroup of $\Z^k,$ it follows that $\mathbf{H}\cong \Z^n$ for some $0\leq n \leq k$.

Similarly,  $\mathbf{K}:=(\T^k)^D= \{\chi\in \widehat{\Z^k}:\,  \chi(A(s^{-1})\mathbf{v})=\chi(\mathbf{v}),\, s \in D, \mathbf{v} \in \Z^k\}$
 is a compact abelian Lie subgroup of $\T^k$. This implies that there exist a finite subgroup ${F}$ of $\T^k$ and a connected closed subgroup $\mathbf{T}$ of $\T^k$ isomorphic to $\T^m$ for some $0 \leq m \leq k$, such that ${F} \cap \mathbf{T}= \{1\}$ and $\mathbf{K}={F}\mathbf{T}$.
 Since $\rank(\mathbf{K})=\rank(\mathbf{T})=m$ and $\rank(\mathbf{H})=n$, our task is to prove that $n=m$.
 Let $\mathbf{W}=\{\mathbf{a}\in\R^k :\,  e^{2\pi i t\mathbf{a}}\in \mathbf{T}, t\in \R\}$.
 In other words, $\mathbf{W}$ is the Lie algebra of $\mathbf{T}$ and its rank is $m$ as well.
If $\chi$ is the character of $\Z^k$ corresponding to $e^{2\pi i t\mathbf{a}}$, with $\mathbf{a} \in \mathbf{W}$ and $t\in \R$,
 the condition $\chi(A(s^{-1})\mathbf{v})=\chi(\mathbf{v})$ is equivalent to
 \begin{equation} \label{eq:ini}e^{2\pi i \langle t\mathbf{a}, \mathbf{v}\rangle}=e^{2\pi i \langle t\mathbf{a}, A(s^{-1})\mathbf{v}\rangle} =e^{2\pi i \langle A(s^{-1})^{T} t\mathbf{a}, \mathbf{v}\rangle},\quad  \mathbf{v} \in \Z^k, \end{equation}
and therefore to
$A(s^{-1})^{T}t\mathbf{a}-t\mathbf{a} \in \Z^k$, for all $s\in D$.
 Since $\Z^k$ is a discrete space, it follows that
  \begin{equation} \label{eq:a} A(s^{-1})^{T}\mathbf{a}-\mathbf{a} =\mathbf{0} \end{equation}
   for all $s\in D$ and $\mathbf{a}\in \mathbf{W} $.  Conversely, if $\mathbf{a}\in \R^k$ satisfies \eqref{eq:a} for all $s\in D$, then equation \eqref{eq:ini}  shows that $\mathbf{a}\in \mathbf{W}$.
   This allows us to conclude that
   \[\mathbf{W}=\{\mathbf{a}\in \R^k:\, A(s^{-1})^{T}\mathbf{a}=\mathbf{a},\, s \in D\}=\{\mathbf{a}\in \R^k:\, A(s)^{T}\mathbf{a}=\mathbf{a},\, s \in D\}.\]
   On the other hand,  since $A(s)$ are integral matrices acting on free abelian groups, one verifies immediately using Gaussian elimination (or the exact sequence for $\mathrm{Tor}^{\bullet}_\Z(\cdot,\R)$) that
   \[\R^n\cong \mathbf{H}\otimes_\Z \R \cong \{\mathbf{v}\in \R^k: A(s)\mathbf{v}-\mathbf{v}=\mathbf{0}, s\in D\}.\]
In view of the previous discussion, we reduced the proof to showing that the vector spaces
  $\mathbf{W}$ and  $\mathbf{H}\otimes_\Z \R $ have the same dimension.

 Let $V$ be a finite dimensional  vector space over a field whose characteristic does not divide $|D|$ and let $\theta: D \to GL(V)$ be a finite group representation. The subspace of invariant vectors is denoted by  $V^D=\{\mathbf{v}\in V: \theta(s)\mathbf{v}=\mathbf{v},\,  s\in D\}$.
 If $T:V \to V$ is a linear map, we denote by $T^*:V^*\to V^*$ its dual map.
 Since the second dual $T^{**}$ identifies naturally with $T$, we can identify $V^D$ with $(V^{**})^D$.
 We claim that $V^D$ and $(V^{*})^D$ have the same dimension. In view of the remark above it suffices to show that
 $\mathrm{dim}(V^D)\leq \mathrm{dim}((V^{*})^D).$
  Choose a linear map $E:V \to V$  such that $E(V)=V^D$ and $E^2=E$. Then $P:V \to V$ defined by
 \begin{equation}\label{eqn:new}P=\frac{1}{|D|}\sum_{t\in D} \theta(t^{-1})E \theta(t) \end{equation}
 satisfies $P(V)=E(V)=V^D$ and
 \[\theta(s)P=P\theta(s)=P=P^2,  \, \,  s\in D.\]
 Indeed, $\theta(s)P=P\theta(s)$ is immediate from \eqref{eqn:new}. To check that $P$ is a projection onto $V^D$ one notes that  $Pw=w$ for $w\in V^D$ since $Ew=w$, and that $P(V)\subseteq V^D$ since $E(V)\subseteq V^D$ and $\theta(t)w=w$ for $w\in V^D$. This last equality also explains why $\theta(s)P=P$.

 Passing to duals we obtain
 \[\theta(s)^*P^*=P^*\theta(s)^*=P^*=(P^*)^2,\, \, s\in D.\]
 This shows that  $P^*(V^*)\subseteq(V^*)^D.$
By elementary linear algebra,   $\rank(P)=\rank(P^*)$ and hence $\mathrm{dim}(V^D)\leq \mathrm{dim}((V^{*})^D)\leq \mathrm{dim}((V^{**})^D)=\mathrm{dim}(V^D).$

Applying all this to $\theta: D \to GL_k(\R)$, we see that the vector spaces
\[\mathbf{H}\otimes _\Z \R=\{\mathbf{v}\in \R^k: A(s)\mathbf{v}=\mathbf{v}, \, s\in D\} \quad \text{and}\quad \mathbf{W}=\{\mathbf{a}\in \R^k:\, A(s)^{T}\mathbf{a}=\mathbf{a},\, s \in D\}\]
  have the same dimension.
 For the second part of the statement, we invoke Proposition \ref{pro:revisit}.
\end{proof}

We need some elements of representation theory and use the book of Kaniuth and Taylor \cite{book:Kaniuth-Taylor} as a basic reference. Let $G$ be a Bieberbach group as in \eqref{eq:bieber}. The  unitary dual of $G$ consists of  equivalence classes of irreducible unitary representations of $G$ and is denoted by $\widehat{G}$. The term character is reserved for one dimensional representations of a group.
 The stabilizer of a character $\chi$ of $N$ is the subgroup of $G$ defined by
$G_\chi=\{g\in G\colon \chi(g^{-1}\cdot g)=\chi(\cdot)\}$. It is clear that $N\subset G_\chi$  and that there is a bijection from $G/G_\chi$ onto the orbit of $\chi$. 
Mackey's theory shows that each irreducible representation $\pi\in \widehat{G}$ is supported by the orbit of some character $\chi\in \widehat{N}$,  in the sense that the restriction of $\pi$ to $N$ is unitarily equivalent to some multiple  of the direct sum of the characters in the  orbit of $\chi$.
\[\pi |_{N }\sim m_\pi \cdot\bigoplus_{g\in G/G_\chi} \chi(g^{-1}\cdot g) \]
where $g$ runs through a set of coset representatives.

For each $\chi\in \widehat{N}$,  denote by $\widehat{G_\chi}^{(\chi)}$ the subset of $\widehat{G_\chi}$ consisting of classes of irreducible representations $\sigma$ of $G_\chi$ such that the restriction of $\sigma$ to $N$ is unitarily equivalent to a multiple of $\chi$. Let $\Omega\subset \widehat{N}$ be a subset which intersects each orbit of $G$ exactly once.
We need the following basic result due to Mackey, see \cite [Thm. 4.28]{book:Kaniuth-Taylor}
\begin{theorem}\label{thm:Mackey} $\widehat{G}=\left\{\mathrm{ind}_{G_\chi}^{\,G}(\sigma)\colon \sigma \in \widehat{G_\chi}^{(\chi)},\,\,\chi\in \Omega\right\}.$
\end{theorem}

The characters of $G$ factor uniquely through $G/[G,G]$ and hence they can be identified with the characters of $H_1(G,\Z)$. Thus,  a Bieberbach group with trivial center has finitely many characters.
\begin{theorem}\label{thm:b}
   Let $G$ be a Bieberbach group with trivial center. If $\omega$ is a character of $G$, then  $\widehat{G}\setminus\{\omega\}$ is a compact open subset of $\widehat{G}.$
\end{theorem}
\begin{proof}
The proof for the Hantzsche-Wendt group $\Gamma$ from \cite{Dad-Pennig-connective},  with $\omega$ the trivial representation,  uses an explicit calculation of the irreducible representations of $\Gamma$ which is not available for an abstract Bieberbach group.
Nevertheless, we can borrow some ideas from there and adapt them to the general situation.

Points are closed  in $\widehat{G}$ so that $\widehat{G}\setminus\{\omega\}$ is open \cite{Dix:C*}.
Since $\widehat{G}$ is compact and satisfies the second axiom of countability \cite{Dix:C*}, we only need to show that $\widehat{G}\setminus \{\omega\}$  is sequentially compact, \cite[p.~138]{book:Kelley}.
Thus it suffices to show that any sequence $(\pi_n)_n$ of points in $\widehat{G}\setminus \{\omega\}$ which converges to $\omega$  has a subsequence which is convergent to a point in $\widehat{G}\setminus \{\omega\}$.
In the terminology of \cite{Dad-Pennig-connective} this means that $\omega$ is a shielded point in $\widehat{G}$.

  Let $(\pi_n)_n$ be a sequence  in $\widehat{G}\setminus \{\omega\}$ which converges to $\omega$.
 Since $\mathrm{dim}(\pi_n)\leq |D|$, after passing to a subsequence, we may arrange that all the representations $\pi_n$ are of the same dimension $m$.
 By  Theorem~\ref{thm:Mackey} there is a sequence $(\chi_n)_n$ in $\Omega$ such that,  up to unitary equivalence,
  $\pi_n =  \mathrm{ind}_{G_{\chi_n}}^{\,G}(\sigma_n)$ with $\sigma_n \in \widehat{G_{\chi_n}}^{(\chi_n)}$. For each character $\chi$ of $N,$ its stabilizer $G_\chi$ is a subgroup of $G$ that contains $N$. Since $D=G/N$ is a finite group, it follows that the set of stabilizers $\{G_\chi:\chi\in \Omega\}$ is finite. Thus, after passing to a subsequence of $(\pi_n)_n$, we may further assume that all stabilizers groups $G_{\chi_n}$ are equal to the same subgroup $L$ with $N \subset L \subset G$.

  We are going to show that if $L=G$, then  $(\pi_n)_n$ cannot converge to $\omega$.
  Indeed, suppose for a moment that  $L=G$. Then each $\chi_n$ is left invariant by the action of $G$ on $\widehat{N}$ so that $\chi_n \in \widehat{N}^G$.

   Since $N^G=Z(G)=\{\mathbf{0}\}$ by assumption, it follows by Proposition \ref{prop} that the group $\widehat{N}^G$ is finite. Therefore, after passing to a subsequence of $(\pi_n)_n,$ we may further assume that all $\chi_n$ are equal to the same character $\chi\in \widehat{N}^G$.
   Then $\pi_n=\sigma_n\in \widehat{G_\chi}^{(\chi)}=\widehat{G}^{(\chi)}$ and hence $\pi_n|_{N}=m\cdot\chi$ for all $n\geq 1$ (recall that $\dim(\pi_n)=m$). Since $\pi_n$ converges to
   $\omega$ and $\pi_n|_{N}=m\cdot\chi$, it follows that $\omega|_{_N}$ is weakly contained in $\chi$. This can happen only if $\chi=\omega|_{N}$.


    One can also argue that $\chi=\omega|_{N}$ as follows.
     Suppose $\chi\neq \omega|_{N}$ so that
   $\chi (h)-\omega(h) \neq 0$ for some $h\in N$. Then  $a=\chi(h)e-h$ is an element of $ C^*(G)$ with $a \in \bigcap_{n\geq 1} \mathrm{Ker}(\pi_n)$ but $a\notin \mathrm{Ker}(\omega)$ since $\omega(a)=\chi(h)-\omega(h)\neq 0$. This contradicts the assumption that $(\pi_n)_n$ converges to $\omega$.

   Every character  of $G$ factors through the finite group $G/[G,G]$.  Thus there are only finitely many characters and their  images are finite groups. Therefore there is a finite index subgroup $K$ of $G$ contained in $N$ on which every character of $G$ is trivial. Since $\pi_n|_{N}=m\cdot\chi=m\cdot\omega|_{N}$, it follows that $\pi_n|_{K}=m\cdot\iota|_{K}$ and hence each $\pi_n$ factors through the homomorphism $G\to G/K$ whose image is finite.  Since the unitary dual of $G/K$ is finite, there are only finitely many distinct terms in the sequence $(\pi_n)_n$ (when viewed as elements of $\widehat{G}$). The points of $\widehat{G}$ are closed and therefore $(\pi_n)_n$ can converge to $\omega$
   only if the sequence is eventually constant and equal to $\omega$. This contradicts the assumption that $\pi_n\in \widehat{G}\setminus\{\omega\}$ for all $n\geq 1$.

Thus it suffices to deal  with the case when all the stabilizers $G_{\chi_n}$ are equal to a fixed subgroup $L \neq G$. Let $r=[G:L]>1$ and choose elements $e_1,...,e_r$ in G such that $e_1$ is the neutral element  and
$G$ is the disjoint union of the cosets $e_iL$. Since all of $(\pi_n)_n$ and hence all of $(\sigma_n)_n$ are of equal dimension,  each
 $\sigma_n$  can be realized on a fixed Hilbert space $H_\sigma=\C^d$, independently of $n$. Then $\pi_n=\mathrm{ind}^{\,G}_L(\sigma_n) : G \to L(H_{\pi_n})$ acts on
\[ H_{\pi_n}=\{\xi:G \to H_\sigma : \, \xi(gh)=\sigma_n(h)^{-1}\xi(g),\, g\in G,\,  h \in L\},\]
by $\pi_n(g)\xi=\xi(g^{-1}\cdot)$.
Each element $\xi\in H_{\pi_n}$ is completely determined on the coset $e_iL$ by its value $\xi(e_i)$ since $\xi(e_ih)=\sigma_n(h)^{-1}\xi(e_i)$ for $h\in L$.
Consider the Hilbert space  $H_\pi= H_\sigma^{\oplus \,r}.$  For each $n$, let $V_n:H_{\pi_n} \to H_\pi$ be the unitary operator defined by
\[V_n(\xi)=(\xi(e_1),...\xi(e_r)).\]
Its adjoint $V_n^*$ maps a vector $(\xi_1,...,\xi_r)\in H_\sigma^{\oplus\,r}$ to a function $\xi:G \to H_{\sigma}$ such that the restriction of $\xi$ to the coset $e_iL$ is given by $\xi(e_ih)=\sigma_n(h)^{-1}\xi_i$ for $h\in L$. We will replace $\pi_n$ by the unitary representation $\rho_n=V_n\pi_n(\cdot)V_n^*$.
Let us observe that $\rho_n(e_r)$ maps $(\xi_1,0,...,0)$ to $(0,...,0,\xi_1)$ for $\xi_1\in H_\sigma$.
Indeed, $V_n^*(\xi_1,0,...,0)=\xi$ where $\xi$ is supported on $L$ and $\xi(h)=\sigma_n(h)^{-1}\xi_1$.
 Then $\pi_n(e_r)\xi=\xi(e_r^{-1}\, \cdot)$ is supported entirely on the coset $e_rL$ and hence
 $V_n\pi_n(e_r)\xi=(0,...,0,\xi_1).$
Let $E:H_\sigma^{\oplus \,r} \to H_\sigma^{\oplus \,r} $ be the orthogonal projection of $H_\pi$ onto its first summand, $E(\xi_1,...\xi_r)=(\xi_1,0,...,0)$.
Since $\rho_n(e_r)(\xi_1,0,...,0)=(0,...,0,\xi_1)$,
it follows that $E\rho_n(e_r)E=0$ for all $n \geq 1$.

  Since $U(H_\pi)$ is compact and the group $G$ is finitely presented (as an extension of finitely presented groups), it follows that
$(\rho_n)_n$ has a subsequence $(\rho_{n_i})_i$ which converges in the point-norm topology to a unitary representation $\rho:G \to U(H_\pi)$. Thus $\lim_{i\to \infty}\|\rho_{n_i}(g)-\rho(g)\|=0$ for all $g\in G$. It follows that $E\rho(e_r)E=0.$
This implies that $\rho$ cannot be a multiple of  $\omega$ since in that case $\rho$ would commute with $E$ and we would have $\|E\rho(e_r)E\|=\|\omega(e_r)E\|=|\omega(e_r)|=1$. Decompose $\rho$ into irreducible representations of $G$. At least one of those, denoted by $\gamma$, must be different from $\omega$.
Since $\lim_{i\to \infty}\|\rho_{n_i}(g)-\rho(g)\|=0$ for all $g\in G$, it follows that $(\rho_{n_i})_i$ and hence $(\pi_{n_i})_i$ converges to $\gamma \in \widehat{G}\setminus \{\omega\}$.
   \end{proof}

   \textbf{Proof of Theorem~\ref{thm:a}.}
   We have seen in the proof of Theorem~\ref{thm:b} (applied for the trivial representation $\iota$) that
   $\iota$ is a shielded point in $\widehat{G}$. By  \cite[Cor. 2.11]{Dad-Pennig-connective}
   if $G$ is any countable amenable group such that $\iota$ is a shielded point in $\widehat{G}$, then $G$ is not connective.

 One may also observe that since virtually abelian groups are type I, the primitive spectrum of $I(G)$ identifies with $\widehat{G}\setminus\{\iota\}$ and hence it is compact. Therefore $I(G)$ is not connective by
 \cite[Prop. 2.7]{Dad-Pennig-connective} as $I(G)\otimes \mathcal{O}_2$ contains a nonzero projection. Here $\mathcal{O}_2$ is the Cuntz algebra. \qed
\section{Connective Bieberbach groups }\label{sec:last}
 In this section we characterize the connective Bieberbach groups.
 Recall that a poly-$\Z$ (or strongly polycyclic) group $G$ is a group which admits a finite increasing series of subgroups
 ${1}=G_0 \subset G_1 \subset \cdots \subset G_k=G$ such that $G_i$ is a normal subgroup of $G_{i+1}$ and $G_{i+1}/G_i \cong \Z$ for all $0 \leq i \leq k-1$.
 For each $i$ we have a split extension
\[
	\xymatrix{
		1 \ar[r] & G_i\ar[r] & G_{i+1} \ar[r] & \Z \ar[r] & 1,
	}
\]
so that $G_{i+1}$ is isomorphic to a semidirect product $G_i\rtimes_{\alpha} \Z$ for some $\alpha\in \Aut(G_i)$.

 \begin{lemma}[Lemma 3.5, \cite{Dad-Pennig-connective}] \label{lem:fin_ext}
Let $m > 1$ and let $G$ and $G'$ be countable discrete groups that fit into a short exact sequence:
\[
	\xymatrix{
		1 \ar[r] & G' \ar[r] & G \ar[r]^-{\pi} & \Z/m\Z \ar[r] & 1\ .
	}
\]
If $G'$ is connective and  the homomorphism $\pi$ factors through $\Z$, then $G$ is also connective.
\end{lemma}
\begin{corollary}\label{cor:ind}
  Let $G$ be a discrete countable group.  Let $\alpha\in \Aut(G)$ be such that $\alpha^m \in \Inn(G)$ for some $m\geq 1$. If $G$ is connective,
  then the semidirect product $G\rtimes_{\alpha}\Z$ is also connective.
\end{corollary}
\begin{proof}
 Since $\beta:=\alpha^m$ is a inner automorphism, $G\rtimes_{\beta}\Z\cong G \times \Z$. It follows that
 $G\rtimes_{\beta}\Z$ is connective since by \cite[Cor. 3.3]{Dad-Pennig-Schneider} direct products of discrete amenable connective groups are connective.
  The group monomorphism
 $G\rtimes_{\beta}\Z \to G\rtimes_{\alpha}\Z$, $(x,k)\mapsto (x,mk)$, $x\in G$, $k \in \Z$,  induces an exact sequence of groups
 \[
	\xymatrix{
		1 \ar[r] & G\rtimes_{\beta}\Z \ar[r] & G\rtimes_{\alpha}\Z \ar[r]^-{\pi} & \Z/m\Z \ar[r] & 1\ .
	}
\]
where $\pi$ is the composition of the  quotient map $\Z\to \Z/m\Z$ with $G\rtimes_{\alpha}\Z \to \Z$.
We conclude that $G\rtimes_{\alpha}\Z $ is connective by applying Lemma~\ref{lem:fin_ext}.
\end{proof}
 \begin{proposition}[Thm. 3.2, \cite{Szczepanski-book}]\label{prop:inner}
 Let $G$ be a Bieberbach group and let $\alpha\in \Aut(G)$.  Then the semidirect product $G\rtimes_\alpha \Z$ is a Bieberbach  group if and only if there exists $m\geq 1$ such that $\alpha^m \in \Inn(G)$.
 \end{proposition}

 Calabi (see \cite{Calabi}, \cite{Szczepanski-book}, \cite{Wolf-spaces}) introduced a reduction method in the study of flat manifolds
 which highlights the central role of the manifolds with the first Betti number zero. In the context of Bieberbach groups
 this translates as follows. If $G$ is a $k$-dimensional Bieberbach group then $H_1(G,\Z)=G/[G,G]$ is a finitely generated abelian group. Thus if $H_1(G,\Z)$ is not a finite group, then there is a surjective map $G \to \Z$.
 Its kernel $G'$ is a $(k-1)$-dimensional Bieberbach group \cite[Prop. 3.1]{Szczepanski-book} and hence $G$ admits a semidirect product decomposition $G\cong G'\rtimes_{\alpha'} \Z$.
 If $H_1(G',\Z)$ is not finite, one can repeat the procedure  and ``peel off" another copy of $\Z$, etc.
\vskip 14pt
 \textbf{Proof of Theorem~\ref{thm:char}}

 (i) $\Rightarrow$ (ii)
  Suppose that $G$ is a connective Bieberbach group. Then all its  subgroups are connective, since connectivity  passes to subgroups \cite{Dad-Pennig-connective}. It follows then by Theorem~\ref{thm:a} and Proposition~\ref{pro:revisit} that every nontrivial  subgroup of $G$ has a nontrivial center.

(ii) $\Leftrightarrow$ (iii)
  A  Bieberbach group $G$ has the property that every nontrivial  subgroup of $G$ has a nontrivial center if and only if it is poly-$\Z$ by \cite[Thm. 23]{Farkas}.

Alternatively, to argue that (ii) $\Rightarrow$ (iii), one can invoke Calabi's method as explained above and write $G$ as an iterated semidirect product
 \[G\cong ((H \rtimes \Z)\rtimes \cdots )\rtimes \Z\]
 where either $H$ is a Bieberbach group with finite first homology group (trivial center) or $H=\{1\}$. By assumption (ii) we see that $H$ must be trivial
 and hence $G$ is poly-$\Z$.

(iii) $\Rightarrow$ (i)
Assume that $G$ is a poly-$\Z$ Bieberbach group of dimension $k\geq 1$. This means that $G$ is constructed iteratively
 by starting with $G_1=\Z$, and then constructing $G_2,...,G_k=G$, where $G_{i+1}=G_i \rtimes_{\alpha_i} \Z$
 for some $\alpha_i \in \Aut(G_i)$. One must have  $\alpha_i^{m_i}\in \Inn(G_i)$   for some $m_i\geq 1$, as a consequence of Proposition~\ref{prop:inner}. Therefore if $G_i$ is connective, then $G_{i+1}$ is also connective. This implication was previously pointed out
  in \cite{Szczepanski} under the stronger assumption that  $\alpha_i$ is a finite order automorphism of $G_i$.
  Since $G_1=\Z$ is connective,  we use Corollary~\ref{cor:ind} to prove by induction that each $G_i$ is connective and hence so is $G$.

(i) $\Leftrightarrow$ (iv) Since $G$ is a virtually abelian group, $C^*(G)$ is a type I $C^*$-algebra.
Thus $\mathrm{Prim}(I(G))=\widehat{I(G)}=\widehat{G}\setminus\{\iota\}$, see \cite{Dix:C*}.
 The desired equivalence follows now from  the characterization of separable nuclear connective $C^*$-algebras in terms of their primitive spectra as explained in the introduction, see \cite{Dad-Pennig-connective} and \cite{Gabe:connectivity}. \qed

\vskip 14pt
 \textbf{Proof of Corollary~\ref{cor}}\quad

 (a)  A connective Bieberbach group is poly-$\Z$ by Theorem~\ref{thm:char} and hence its quotient $D$ must be solvable.

 (b) Let $\G$ be the class of all finite groups whose Sylow subgroups are all cyclic.
If $D$ is in $\G$, so are the normal subgroups and the quotients of $D$.
We prove by induction on $k$ that a $k$-dimensional Bieberbach group with holonomy in $\G$ is poly-$\Z$ and then apply Theorem~\ref{thm:char} to conclude that $G$ is connective. If $k=1$ then $G\cong \Z$ and we are done.
 Suppose now that $k>1$. By assumption,  $D$  belongs to the class $\G$,  and hence it  is not primitive,
 as explained in Section~\ref{sec:2}.
 Therefore $H_1(G,\Z)$ is not finite and as seen earlier in the proof of Theorem~\ref{thm:char}, $G\cong G' \rtimes \Z$ where $G'$ is a Bieberbach group of rank $k-1$.  The following diagram with exact rows and columns
 \begin{equation*}\label{eq:Calabi}
\xymatrix{
	& 0\ar[d]& 0 \ar[d] & 0 \ar[d] \\
	0 \ar[r] & N\cap G'  \ar[d] \ar[r] & G' \ar[r] \ar[d]& D_0  \ar[r] \ar[d] & 0\\
 	0 \ar[r] & N \ar[d] \ar[r] & G \ar[r] \ar[d] & D \ar[r] \ar[d] & 0 \\
 	0 \ar[r] & N/N\cap G'\ar[d] \ar[r] & \Z\ar[r] \ar[d] & D/D_0 \ar[r] \ar[d] & 0 \\
	& 0& 0 & 0
}
\end{equation*}
shows that the quotient of $G'$ by $N\cap G'\cong \Z^{k-1}$, denoted by $D_0$, is isomorphic to  a normal subgroup of $D$. The finite group $D_0$ is not necessarily the holonomy group of $G'$ since $N\cap G'$ is not always maximal abelian in $G'$. However, by
\cite[Thm. 3.1]{Vasquez}, the centralizer of $N\cap G'$ in $G'$  is a maximal abelian normal subgroup $N'\cong \Z^{k-1}$ of $G'$.
 It follows that the holonomy group of $G'$ is $D'\cong G'/N'$ and moreover $D'$  is isomorphic to
 a quotient of $D_0$ because $N\cap G'\subset N'$.  Since $D$ is in $\G$, it follows that  $D'$ belongs to $\G$, as explained at the beginning of the proof. By the induction hypothesis, $G'$ is poly-$\Z$ and therefore so is $G$, since $G\cong G' \rtimes \Z$. \qed

 (c)   We have seen, as a consequence of
Theorem~\ref{thm:char}, that   a Bieberbach group $G$ is connective $\Leftrightarrow$ $G$ is locally indicable $\Leftrightarrow$ $G$ is diffuse. By \cite[Thm.~3.5 (iii)]{Diffuse}, if a finite group $D$ is solvable and has a non-cyclic Sylow group, then $D$ can be realized as the holonomy of both a diffuse  group $G_1$ and a  non-diffuse group $G_2$. \qed

\begin{remark} In view of the previous discussion, one can also formally derive both parts (a) and (b) of Corollary~\ref{cor}   from Theorem~\ref{thm:char} and \cite[Thm.~3.5 (i), (ii)]{Diffuse}.
Nevertheless,  we included a direct proof,  as a mean to review  Calabi's method on which the main result of our paper is based.

In \cite{GLS}, the authors find all Bieberbach groups up to dimension six that are non-diffuse and hence not connective. There are $38,746$ six-dimensional Bieberbach groups, out of which  $19,256$ (almost a half) are not connective.

There are no general classification results for Bieberbach groups.  To give an idea of the complexity of this question, let us mention that the number of non-isomorphic $k$-dimensional Bieberbach groups with   holonomy $\Z/2 \oplus \Z/2$ and finite first homology group grows as least as fast as $C k^5$ for some $C>0$, \cite{Rossetti-T}.

\end{remark}

  \vskip 14pt
 \textbf{Acknowledgements}
 We thank Nansen Petrosyan for calling our attention to the references \cite{GLS} and \cite{Diffuse}.
 We also thank the referee for a useful suggestion concerning the proof of Proposition \ref{prop}.
\bibliographystyle{abbrv}

\end{document}